\documentclass[11pt,UTF8]{article}

\usepackage{bbm,dsfont,yfonts}
\usepackage[all,cmtip]{xy}
\usepackage{amsfonts,amssymb,amsmath,amscd,amsthm,mathtools}
\usepackage{mathrsfs}
\usepackage{latexsym}
\usepackage{multicol}
\usepackage{color}
\usepackage{hyperref}
\usepackage{verbatim,enumitem}
\usepackage{tikz-cd}
\usepackage{booktabs} 
\input diagxy
\usepackage[total={150mm,230mm}]{geometry}

\DeclareMathOperator{\with}{\&}

\theoremstyle{plain}
  \newtheorem{theorem}{Theorem}[section]
  \newtheorem{lemma}[theorem]{Lemma}
  \newtheorem{prop}[theorem]{Proposition}
  \newtheorem{cor}[theorem]{Corollary}
\theoremstyle{definition}
  \newtheorem{definition}[theorem]{Definition}
  
  \newtheorem{example}[theorem]{Example}

\newcommand{\bv}{\bigvee}

\allowdisplaybreaks

\begin{document}

\title{Formal balls of  ${\sf Q}$-categories}

\author{Xianbo Yang,  Dexue Zhang \\ {\small School of Mathematics, Sichuan University, Chengdu, China}\\  {\small  xianboyang@outlook.com, dxzhang@scu.edu.cn}  }
\date{}
\maketitle

\begin{abstract}The construction of the formal ball model for metric spaces due to  Edalat and Heckmann was generalized to {\sf Q}-categories by Kostanek and Waszkiewicz.  This paper concerns the influence of the structure of the quantale {\sf Q} on the connection between Yoneda completeness of ${\sf Q}$-categories and directed completeness of their sets of formal balls. In the case that {\sf Q} is the interval $[0,1]$ equipped with a continuous t-norm \&, it is shown that in order that  Yoneda completeness of each {\sf Q}-category be equivalent to directed completeness of its set of formal balls, a necessary and sufficient condition is that the t-norm \& is Archimedean.
 \vskip 1pt

\end{abstract}


\section{Introduction}
The formal ball construction  is a basic tool for quasi-metric spaces, as demonstrated in \cite{Goubault,Goubault2019,Goubault-Ng}.
Goubault-Larrecq and Ng  \cite{Goubault-Ng} argued  that ``formal balls are the essence of quasi-metric space''. Formal balls were first introduced in \cite{WS} for metric spaces, later extended to quasi-metric spaces in \cite{Ru,SV2005}, then to the general setting of ${\sf Q}$-categories in \cite{KW2011}, where ${\sf Q}$ is a commutative and unital quantale.

Metric properties of a quasi-metric space $X$ are closely related with the order structure of its set   $\mathbf{B}X$ of formal balls.  A typical example says, a quasi-metric space $X$ is  continuous and  Yoneda complete  in the sense of \cite{KW2011,Was2009} if and only if   $\mathbf{B}X$ is a continuous dcpo (directed complete partially ordered set) \cite{Goubault,Goubault-Ng}.  Another example in this vein, which motivates this paper, asserts that a quasi-metric space $X$ is Yoneda complete if and only if $\mathbf{B}X$ is a dcpo.     This result was first proved in \cite{EH98} for metric spaces; then in \cite{Ali,KW2011} for quasi-metric spaces. In fact, Kostanek and Waszkiewicz \cite{KW2011} proved the conclusion for ${\sf Q}$-categories with ${\sf Q}$ being a special kind of \emph{value quantale}, not only for quasi-metric spaces. But, the requirements imposed on the quantale in \cite{KW2011} are so strong  that if ${\sf Q}$ is the interval $[0,1]$ together with a continuous t-norm,  up to isomorphism  there is  only one t-norm, the product t-norm,  that satisfies the requirements.

Thus, for a general quantale $\sf Q$, even for the quantale obtained by endowing the interval $[0,1]$ with a continuous t-norm, the question remains open  whether we have an equivalence between Yoneda completeness of a $\sf Q$-category  and directed completeness of its set of formal balls.

In this paper we show that the answer is negative. Actually, the answer depends on the structure of the quantale {\sf Q}, as we shall see.  Corollary \ref{maincor} shows that, in the case that  ${\sf Q}$ is the interval $[0,1]$ equipped with a continuous t-norm,  the equivalence holds if  and only if  the t-norm is Archimedean.

\section{Preliminaries}

A commutative and unital  quantale (a quantale for short) \cite{Rosenthal1990} \[{\sf Q}=({\sf Q},\with,k)\] is a commutative monoid with $k$ being the unit, such that the underlying set ${\sf Q}$ is a complete lattice (with a top element $1$ and a bottom element $0$), and that the multiplication $\with$ distributes over arbitrary joins in the sense that \[p\with\Big(\bigvee_{i\in I}q_i\Big)=\bigvee_{i\in I}p\with q_i\] for all $p, q_i\in{\sf Q}$ $(i\in I)$.   If  the unit $k$ is the top element, then we say that ${\sf Q}$ is \emph{integral}. For all elements $p,q$ of an integral quantale, we always have $p\with q\leq p\wedge q$.

Given a quantale ${\sf Q}=({\sf Q},\with,k)$, the  multiplication $\&$ determines a binary operator $\rightarrow$,  called the implication  corresponding to $\&$, via the adjoint property:
\[p\with q\leq r\iff q\leq p\rightarrow r.\]

Typical examples of quantales include (i) $(H,\wedge,1)$, where $H$ is a complete Heyting algebra; (ii) Lawvere's quantale  $([0,\infty]^{\rm op},+,0)$; and   (iii)   $([0,1],\with,1)$, where $\with$ is a left continuous t-norm. Actually, a left continuous t-norm on $[0,1]$  \cite{KMP00} is just a binary operation $\with\colon[0,1]\times[0,1]\longrightarrow[0,1]$ such that $([0,1],\with,1)$ is a quantale.

A continuous t-norm on $[0,1]$ is a left continuous t-norm that is  continuous with respect to the usual topology.   We refer to the monograph \cite{KMP00} for   continuous t-norms. Three basic   continuous t-norms and their implication operators are listed below:
\begin{enumerate}[label={\rm(\roman*)}] \setlength{\itemsep}{0pt}
\item  The G\"{o}del t-norm:  \[ x\with_G y= \min\{x,y\}; \quad x\rightarrow y=\begin{cases}
		1 &x\leq y,\\
		y &x>y.
		\end{cases} \] 

\item  The product t-norm:   \[ x\with_P y=xy; \quad x\rightarrow y=\begin{cases}
		1 &x\leq y,\\
		y/x &x>y.
		\end{cases} \] 

\item  The {\L}ukasiewicz t-norm: \[ x\with_{\text \L} y=\max\{0,x+y-1\}; \quad x\rightarrow y=\min\{1-x+y,1\}. \]
	\end{enumerate}

A continuous t-norm on $[0,1]$ is   \emph{Archimedean}, if for all $x,y\in(0,1)$  there is some integer $n$ such that $x^n<y$, where $$x^n=\underbrace{x\with x\with\cdots\with x}_{n~\text{times}}.$$ It is not hard to see that a continuous t-norm is Archimedean if and only if it has no idempotent element  other than $0$ and $1$. It is well-known (see e.g.  \cite{KMP00})  that if $\with$ is a continuous Archimedean t-norm, then the quantale  $([0,1],\with,1)$ is either isomorphic to  $([0,1],\with_{\text \L},1)$ or to $([0,1],\with_{P},1)$. In other words, up to isomorphism there are precisely two Archimedean continuous t-norm on $[0,1]$: the product t-norm and the {\L}ukasiewicz t-norm.
\begin{definition}\cite{HST2014,Lawvere1973,Wagner97}
Let ${\sf Q}$ be a quantale.  A  ${\sf Q}$-category  consists of a set $X$ and a map $o\colon X\times X\longrightarrow {\sf Q}$  such that
 \[k\leq o(x,x), \quad
o(y,z)\with  o(x,y)\leq o(x,z)\]
for all $x,y,z\in X$. As usual, we   write $X$ for the pair $(X, o)$ and   $X(x,y)$ for $o(x,y)$ if no confusion would arise.\end{definition}

If ${\sf Q}$ is the Boolean algebra $\{0,1\}$, then a ${\sf Q}$-category is exactly a preordered set; that is, a set together with a reflexive and transitive relation. If ${\sf Q}$ is Lawvere's quantale $([0,\infty]^{\rm op},+,0)$, then a ${\sf Q}$-category is exactly a generalized metric space in the sense of Lawvere \cite{Lawvere1973}; such a  ${\sf Q}$-category is also known as a pseudo-quasi-metric space (with distance allowed to be infinite).

Let $X$ be a ${\sf Q}$-category. A \emph{formal ball} of $X$ is a pair $(x,r)$ with $x\in X$ and $r\in{\sf Q}$, $x$ is called the center and $r$ the radius.
For  formal balls  $(x,r)$ and $(y,s)$, define \[(x,r)\leq_{\mathbf{B}X} (y,s) \quad\text{if}~ r\leq s\with X(x,y).\] Then $\leq_{\mathbf{B}X}$ is a reflexive and transitive relation, hence a preorder. We write $\mathbf{B}X$ for the set of  formal balls of $X$ endowed with the preorder $\leq_{\mathbf{B}X}$. We often omit the subscript  if it causes no confusion. We note that in this paper  the radius $r$ of a formal ball $(x,r)$ is allowed to be the bottom element of {\sf Q}.

A net $(x_\lambda)_{\lambda\in D}$ in a ${\sf Q}$-category  $X$  is \emph{forward Cauchy} \cite{FSW1996,Wagner97} if \[\bigvee_\lambda\bigwedge_{\lambda\leq\gamma\leq\mu}X(x_\gamma,x_\mu)\geq k.\]
An element $a$ of $X$ is a \emph{Yoneda limit}   \cite{FSW1996,Wagner97} of  $(x_\lambda)_{\lambda\in D}$  if  for all $y\in X$, \[X(a,y)=\bigvee_\lambda\bigwedge_{\mu\geq\lambda}X(x_\mu,y).  \]

Yoneda limits of forward Cauchy nets can be characterized  as colimits of forward Cauchy  weights. A \emph{weight} of a $\sf Q$-category $X$ is a map $\phi\colon X\longrightarrow{\sf Q}$  such that $ \phi(y)\with X(x,y)\leq \phi(x)$  for all $x,y\in X$. For each forward Cauchy net $(x_\lambda)_{\lambda\in D}$ of $X$,  the map \[\phi\coloneqq\bigvee_\lambda\bigwedge_{\mu\geq\lambda}X(-,x_\mu) \]  is a weight of $X$, such a weight is said to be \emph{forward Cauchy} \cite{LZZ2020}.
Forward Cauchy weights of a $\sf Q$-category $X$ are also known as  \emph{ideals}  in the literature, see e.g. \cite{FSW1996}. An element $a$ of $X$ is a \emph{colimit} of a weight $\phi$ \cite{HST2014,FSW1996} if  for all $y\in X$, \[X(a,y)=\bigwedge_{x\in X}(\phi(x)\rightarrow X(x,y)).\]
\begin{prop} {\rm (\cite[Lemma 46]{FSW1996})}
\label{yoneda limit as colimits} Let $(x_\lambda)_{\lambda\in D}$ be a forward Cauchy net of a ${\sf Q}$-category  $X$. Then, an element $a$ of $X$ is a Yoneda limit of $(x_\lambda)_{\lambda\in D}$ if and only if $a$ is a colimit of the forward Cauchy weight $\phi=\bigvee_\lambda\bigwedge_{\mu\geq\lambda}X(-,x_\mu)$. Therefore, a ${\sf Q}$-category $X$  is   Yoneda complete  if and only if   every forward Cauchy weight of $X$  has a colimit. \end{prop}

Given a $\sf Q$-category $X$,  Yoneda completeness  of $X$ is closely related with directed completeness of the set $\mathbf{B}X$ of its formal balls. As mentioned before, when $\sf Q$ is Lawvere's quantale $([0,\infty]^{\rm op},+,0)$, a $\sf Q$-category (i.e., a generalized metric space) is Yoneda complete if and only if its set of formal balls is directed complete \cite{Ali,EH98,KW2011}.

In the case that ${\sf Q}$  is a continuous and integral quantale, Proposition \ref{yoneda limit as colimits} and Lemma \ref{characterization of Cauchy ideal} below explain to some extent why Yoneda completeness of a $\sf Q$-category and directed completeness of its set of formal balls are  closely related.


Before proceeding on, we recall the notion of continuous lattices first.
Let $a,b$ be elements of a partially ordered $P$. We say that $a$ is \emph{way below} $b$, in symbols $a\ll b$, if for each directed set $D$ of $P$ with a join,  \[b\leq\textstyle{\bigvee} D \implies \exists\thinspace d\in D, a\leq d.\] A \emph{continuous lattice} \cite{Gierz2003} is a complete lattice $L$ for which every element is the join of  elements way below it; that is, $a=\bigvee\{x\in L\mid x\ll a\}$. The interval $[0,1]$ is clearly a continuous lattice.
A \emph{continuous quantale} is a quantale for which the  underlying lattice is  continuous.

\begin{lemma}\label{characterization of Cauchy ideal} Let ${\sf Q}$ be a continuous and integral quantale. Then  for each weight $\phi$ of a ${\sf Q}$-category $X$, the following are equivalent: \begin{enumerate}[label=\rm(\arabic*)]\setlength{\itemsep}{0pt} \item $\phi$ is   forward Cauchy.
\item   $\phi$   satisfies the following conditions:  \begin{enumerate}[label=\rm(\roman*)]\setlength{\itemsep}{0pt} \item $\bigvee_{x\in X}\phi(x)=1$;
\item If $r\ll 1$ and $s_i\ll\phi(x_i)~ (i=1,2)$, then  there exists  $x\in X$ such that $r\ll\phi(x)$ and that  $s_i\ll X(x_i,x) ~ (i=1,2)$.
\end{enumerate}
\item There is a directed subset $(x_\lambda,r_\lambda)_{\lambda\in D}$   of $\mathbf{B}X $ such that $\bigvee_{\lambda\in D}r_\lambda=1$ and that  $\phi=\bigvee_{\lambda\in D}\bigwedge_{\mu\geq \lambda} X(-,x_\mu)$. \end{enumerate} \end{lemma}

\begin{proof}That  (3) implies (1)  follows  immediately from Lemma \ref{directed set is forward Cauchy} below. The equivalence $(1)\Leftrightarrow(2)$ is contained in \cite[Lemma 6.3]{LZ07}; the implication $(2)\Rightarrow(3)$ is also proved there implicitly. So, here we only write down the construction of the directed subset.  Suppose that $\phi$ satisfies the conditions (i) and (ii).    Let   \[\mathrm{B}\phi=\{(x,r)\in \mathbf{B}X \mid  r\ll\phi(x)\}.  \]
Then $\mathrm{B}\phi$ is a directed subset  of $\mathbf{B}X $ that satisfies the requirement. \end{proof}

\begin{lemma}\label{directed set is forward Cauchy} Let ${\sf Q}$ be an integral quantale; let $X$ be a ${\sf Q}$-category and $(x_\lambda,r_\lambda)_{\lambda\in D}$ be a directed subset of $\mathbf{B}X $. If $\bigvee_{\lambda\in D}r_\lambda=1$, then $(x_\lambda)_{\lambda\in D}$ is a forward Cauchy net in $X$. \end{lemma}

\begin{proof} Since $r_\lambda \leq r_\mu\with X(x_\lambda,x_\mu)$ whenever  $\lambda\leq\mu$, it follows that \[1=\bigvee_{\lambda\in D}r_\lambda\leq\bigvee_{\lambda\in D}\bigwedge_{\lambda\leq\mu}r_\mu \leq \bigvee_{\lambda\in D}\bigwedge_{\lambda\leq\mu\leq\gamma}X(x_\mu,x_\gamma),\] hence $(x_\lambda)_{\lambda\in D}$ is   forward Cauchy. \end{proof}

\section{Counterexamples}

In this section we give two examples to show that for a general quantale $\sf Q$, Yoneda completeness of a $\sf Q$-category may fail to be equivalent to directed completeness of its set of formal balls.

\begin{example}\label{Example-G} This example   presents a  ${\sf Q}$-category that is Yoneda complete, but its set of formal balls is not directed complete.

Let $\with$ be a continuous non-Archimedean t-norm and let ${\sf Q}$ be the quantale $([0,1],\with,1)$. Since $\with$ is continuous and non-Archimedean, there is some $b\in(0,1)$ such that $b\with b=b$. Let $X=(0,b)$. Define a ${\sf Q}$-category structure on $X$ by \[X(x,y)=\begin{cases}1 & x=y,\\
\min\{x\rightarrow y,y\rightarrow x\} & x\not=y.\end{cases}\]

Since $X(x,y)\leq b$ whenever $x\not=y$, every forward Cauchy net of $X$ is eventually constant, so $X$ is Yoneda complete. It remains to show that $\mathbf{B}X$ is not directed complete. To this end, pick a strictly increasing sequence $(x_n)_{n\geq1}$  in $(0,b)$ that converges to $b$. For each $n$ let $r_n=x_n$. We claim that the subset $(x_n,r_n)_{n\geq1}$ of $\mathbf{B}X$ is directed and has no join.

Since $\with$ is a continuous t-norm, then  for all $x,y\in[0,1]$ we have $$x\with(x\rightarrow y) =\min\{x,y\}. $$  For all $n\leq m$, since  \begin{align*}r_n&=x_n =x_m\with(x_m\rightarrow x_n) = r_m\with X(x_n,x_m),\end{align*} then $(x_n,r_n)\leq(x_m,r_m)$, hence $(x_n,r_n)_{n\geq1}$ is directed.

Next  we show that $(x_n,r_n)_{n\geq1}$ does not have a join. Suppose on the contrary that $(x,r)$ is a join of $(x_n,r_n)_{n\geq1}$. Since $x<b$, there is some $n_0$ such that $x<x_m$ for all $m\geq n_0$.   Since $(x,r)$ is an upper bound  of $(x_n,r_n)_{n\geq1}$, for each $m\geq n_0$, we have \[r_m\leq r\with X(x_m,x)\leq x_m\rightarrow x,\] this is impossible since the left side tends to $b$, while the right side tends to $x$.\end{example}

\begin{example}\label{exmp3} This example   presents a ${\sf Q}$-category  that is not Yoneda complete, but its set  of formal balls is   directed complete.

Let ${\sf Q}$ be the quantale $([0,1],\with,1)$, where $\with$ is the G\"{o}del t-norm $\min$. Let $X=\{1\}\cup\{1-1/n\mid n\geq2\}$. Define a ${\sf Q}$-category structure on $X$ by $$X(x,y)=\begin{cases}1 & x=y,\\ 1/3 & x=1, y\not=1,\\ \min\{x,y\} &{\rm otherwise}.  \end{cases}$$  We claim that  $X$ is not Yoneda complete, but $\mathbf{B}X$   is directed complete.

For each $n\geq 2$, let $x_n=1-1/n$. It is readily verified that the sequence $(x_n)_{n\geq2}$ is  forward Cauchy and  has no Yoneda limit, so $X$ is not Yoneda complete.  It remains to check that $\mathbf{B}X$   is directed complete. Given a directed subset $D$ of $\mathbf{B}X$, we write $D$ as a net  $(x_\lambda,r_\lambda)_{\lambda\in D}$ indexed by itself. Since $r_\lambda\leq r_\mu\with X(x_\lambda,x_\mu)$ whenever $\lambda\leq \mu$,  the net $(r_\lambda)_{\lambda\in D}$ is monotone. Let $r=\bv_{\lambda\in D}r_\lambda$.
Now we proceed with two cases.

Case 1. The net $(x_\lambda)_{\lambda\in D}$ is eventually constant; that means, there is some $a\in X$ and some $\lambda\in D$ such that $x_\mu=a$ whenever $\mu\geq\lambda$.  In this case   $(a,r)$ is a join of $D$.

Case 2. The net $(x_\lambda)_{\lambda\in D}$ is not eventually constant. Then for each $\lambda$ there is some $\mu\geq \lambda$ such that $x_\lambda\not=x_\mu$, hence $$r_\lambda\leq r_\mu\with X(x_\lambda,x_\mu)\leq \min\{ r_\mu,x_\lambda,x_\mu\}\leq x_\lambda.$$ Let  $a=\min\{x\in X\mid r\leq x\}.$
Then $(a, r)$ is a join of $D$. \end{example}

A ${\sf Q}$-category $X$ is   \emph{Smyth complete} if it is Cauchy complete (see \cite{HST2014,Lawvere1973} for definition) and all forward Cauchy weights of $X$ are Cauchy. The notion of Smyth completeness originated in \cite{Smyth88}. For Smyth completeness of quasi-metric spaces, the reader is referred to \cite{Goubault,KS2002,RV2010}. The postulation of Smyth complete ${\sf Q}$-categories given here is based on the characterization of Smyth complete quasi-metric spaces in \cite[Section 6]{Liwei2018}.

Consider the ${\sf Q}$-category $X$ in Example \ref{Example-G}. Since every forward Cauchy net of $X$ is eventually constant,   $X$ is Yoneda complete and Smyth complete, hence continuous in the sense of \cite{KW2011,Was2009}. So, in contrast to  the situation for quasi-metric spaces \cite[Theorem 3.2]{RV2010},  for a general quantale {\sf Q}, the set of formal balls of a Smyth complete ${\sf Q}$-category may fail to be directed complete.

Example \ref{Example-G} also shows that for a directed subset $(x_\lambda,r_\lambda)_{\lambda\in D}$  of $\mathbf{B}X$, the net $(x_\lambda)_{\lambda\in D}$ of $X$ need not be forward Cauchy. 

\begin{prop} \label{Archi vs FC} Let $\with$ be a continuous t-norm on $[0,1]$ and let ${\sf Q}=([0,1],\with,1)$. The following are equivalent: \begin{enumerate}[label=\rm(\arabic*)]\setlength{\itemsep}{0pt} \item   $\with$ is   Archimedean. 
\item For each ${\sf Q}$-category $X$ and each directed subset $(x_\lambda,r_\lambda)_{\lambda\in D}$  of $\mathbf{B}X $ with   some $r_\lambda>0$, $(x_\lambda)_{\lambda\in D}$ is a forward Cauchy net. \end{enumerate} \end{prop}

\begin{proof} $(1)\Rightarrow(2)$ We'll make use of the following fact about continuous  Archimedean  t-norms: if $0<r\leq s\with t$, then $t\geq s\rightarrow r.$

Assume that  $(x_\lambda,r_\lambda)_{\lambda\in D}$ is a directed subset  of $\mathbf{B}X $  and, without loss of generality,  assume that   $r_\lambda>0$ for all $\lambda\in D$. Since $r_\lambda\leq r_\mu\with X(x_\lambda,x_\mu)$ whenever $\lambda\leq\mu$, then  $X(x_\lambda,x_\mu)\geq r_\mu\rightarrow r_\lambda$ whenever  $\lambda\leq\mu$. Since $(r_\lambda)_{\lambda\in D}$ converges to its join and the implication operator of an Archimedean   continuous t-norm is continuous except possibly at $(0,0)$, it follows that $X(x_\lambda,x_\mu)$ tends to $1$, so $(x_\lambda)_{\lambda\in D}$ is forward Cauchy.

$(2)\Rightarrow(1)$ Suppose on the contrary that $\with$ is non-Archimedean. Consider the ${\sf Q}$-category $X$  in Example \ref{Example-G}. Then the  subset $(x_n,r_n)_{n\geq1}$ of $\mathbf{B}X$ given there is    directed, but  $(x_n)_{n\geq2}$ is not forward Cauchy, a contradiction. \end{proof}

\section{The main result}
In order to state the main result, we still need two notions.

Let ${\sf Q}=({\sf Q},\with,k)$ be a quantale. We say that $\with$ \emph{distributes over non-empty meets} if  \[p\with\Big(\bigwedge_{i\in I}q_i\Big)=\bigwedge_{i\in I}p\with q_i\] for any $p\in \sf Q$ and any non-empty subset $(q_i)_{i\in I}$ of {\sf Q}.  It is clear that any continuous t-norm on $[0,1]$ distributes over non-empty meets.

\begin{definition}\label{defn of R}
Let ${\sf Q}$ be a continuous and integral quantale.  We say that a ${\sf Q}$-category $X$ has  property {\rm(R)}, if for each pair $(s,t)$ of elements of ${\sf Q}$ with $0<s\leq t$, there is some $r\ll1$ such that for all $x,y\in X$  and   all $r'\geq r$, we always have \[(x,t\with r')\leq (y,s)\iff(x,r')\leq(y,t\rightarrow s).\] \end{definition}


Corollary \ref{characeterizing R} below provides a characterization of $\sf Q$-categories with property (R) in the case that the quantale $\sf Q$ is the interval $[0,1]$ together with a continuous t-norm. By this characterization it is easy to find $\sf Q$-categories with or without property (R). Now we  present the main result of this paper.

\begin{theorem}\label{main}Let ${\sf Q}$ be a continuous and integral quantale such that $\with$ distributes over non-empty meets;  let $X$ be a  ${\sf Q}$-category. \begin{enumerate}[label=\rm(\roman*)]\setlength{\itemsep}{0pt} \item If $X$ is  Yoneda complete, then each directed subset $(x_\lambda,r_\lambda)_{\lambda\in D}$ of $\mathbf{B}X $ with $\bigvee_{\lambda\in D}r_\lambda=1$ has a join. \item If $X$ has property {\rm(R)} and  $\mathbf{B} X$ is directed complete, then $X$ is Yoneda complete. \end{enumerate}  \end{theorem}

Before proving Theorem \ref{main}, we make some preparations.

\begin{lemma}\label{join of directed set in BX}Let ${\sf Q}$ be a continuous and integral quantale such that $\with$ distributes over non-empty meets;   let $X$ be a ${\sf Q}$-category and let $(x_\lambda,r_\lambda)_{\lambda\in D}$ be a directed subset    of $\mathbf{B}X $. If $(x_\lambda)_{\lambda\in D}$ is a forward Cauchy net with $x$ being a Yoneda limit, then $(x,r)$ is a join of the directed set $(x_\lambda,r_\lambda)_{\lambda\in D}$, where $r=\bigvee_{\lambda\in D}r_\lambda$. \end{lemma}

\begin{proof}The proof is   a slight improvement of that for  Lemma 7.7 in \cite{KW2011}.

 First, we show that $(x,r)$ is an upper bound of $(x_\lambda,r_\lambda)_{\lambda\in D}$; that is, $r_\lambda\leq r\with X(x_\lambda,x)$ for all $\lambda\in D$.

For each $\lambda\in D$ and each $\epsilon\ll r_\lambda$, since $\with$ distributes over non-empty meets and \[1=X(x,x)= \bigvee_\delta\bigwedge_{\mu\geq\delta}X(x_\mu,x),\]   there is some $\delta\in D$ such that $\epsilon\leq r_\lambda\with X(x_\mu,x)$ whenever $\mu\geq\delta$.  Thus, for all $\mu\geq\lambda,\delta$, we have \begin{align*}\epsilon&\leq r_\lambda\with X(x_\mu,x) \leq r_\mu\with X(x_\lambda,x_\mu)\with X(x_\mu,x) \leq r\with X(x_\lambda,x).\end{align*} By   arbitrariness of $\epsilon$  we obtain that $r_\lambda\leq r\with X(x_\lambda,x)$.

Next  we show that $(x,r)\leq (y,s)$ for any upper bound $(y,s)$ of $(x_\lambda,r_\lambda)_{\lambda\in D}$.
Since $(y,s)$ is an upper bound  of $(x_\lambda,r_\lambda)_{\lambda\in D}$,
then  $r_\lambda\leq r_\mu\leq s\with X(x_\mu,y)$ whenever $\lambda\leq\mu$, hence \begin{align*} r&=\bigvee_{\lambda\in D}r_\lambda  \leq \bigvee_{\lambda\in D}\bigwedge_{\mu\geq\lambda}s\with X(x_\mu,y)  = s\with X(x,y),\end{align*} which shows that $(x,r)\leq (y,s)$, as desired. \end{proof}

\begin{prop}\label{directed set with radius 1} Let ${\sf Q}$ be a continuous and integral quantale such that $\with$ distributes over non-empty meets. If $X$ is a Yoneda complete ${\sf Q}$-category, then every directed subset $(x_\lambda,r_\lambda)_{\lambda\in D}$ of $\mathbf{B}X $ with $\bigvee_{\lambda\in D}r_\lambda=1$ has a join. \end{prop}

\begin{proof}This follows directly from Lemma \ref{directed set is forward Cauchy} and  Lemma \ref{join of directed set in BX}. \end{proof}
\begin{lemma}\label{expansion stable} Let ${\sf Q}$ be a continuous and integral quantale. Then for all $x$ and $y$ of a  ${\sf Q}$-category $X$  with property {\rm(R)}, the following conditions are equivalent: \begin{enumerate}[label=\rm(\arabic*)]\setlength{\itemsep}{0pt}
 \item $(x,1)\leq(y,1)$. \item $(x,s)\leq(y,s)$ for all $s\not=0$. \item $(x,s)\leq(y,s)$ for some $s\not=0$.  \end{enumerate}\end{lemma}

\begin{proof}It suffices to check $(3)\Rightarrow(1)$. Since $0<s\leq s$, there is some $r\ll1$ such that  \[(x,s\with r')\leq (y,s)\iff(x,r')\leq(y,s\rightarrow s)\] for all $r'\geq r$. Putting $r'=1$ gives that $(x,1)\leq(y,1)$. \end{proof}

\begin{prop} \label{Archi vs reciprocal}  Let $\with$ be a continuous t-norm on $[0,1]$ and let ${\sf Q}=([0,1],\with,1)$. Then,   $\with$ is Archimedean if and only if every ${\sf Q}$-category has property {\rm(R)}.  \end{prop}

\begin{proof}For sufficiency we need to show that $\with$ has no nontrivial idempotent element. For this it suffices to show that for all $s\not=0$ and $q\in[0,1]$, if $s\leq s\with q$ then $q=1$. Consider the $\sf Q$-category  $X=\{x,y\}$ with $X(x,x)=X(y,y)=1$ and $X(x,y)=q=X(y,x)$. Since $(x,s)\leq (y,s)$ and $X$ has property   {\rm(R)}, then $(x,1)\leq(y,1)$, hence $1\leq 1\with X(x,y)=q$.

As for necessity, assume that $\with$ is a continuous Archimedean t-norm. Then $\with$ is either isomorphic to the \L ukasiewicz t-norm or to the product t-norm. In the following we check the conclusion for the case that $\with$ is isomorphic to the  \L ukasiewicz t-norm, leaving the other case  to the reader.

Without loss of generality, we assume that $\with$ is, not only isomorphic to, the  \L ukasiewicz t-norm; that is, $$x\with y=\max\{0,x+y-1\}.$$  Suppose that $X$ is a ${\sf Q}$-category and $0<s\leq t$. If $t=1$,   it is trivial  that \[(x,t\with r')\leq (y,s)\iff(x,r')\leq(y,t\rightarrow s) \]  for all $r'>0$, so each $r>0$ satisfies the requirement. If $t<1$, pick $r\in(1-t,1)$. Then $r\ll1$ and for all $r'\geq r$, \begin{align*}  (x,t\with r')\leq (y,s)  &\iff r'+t-1\leq s+X(x,y)-1 \\ &\iff r'\leq s-t+X(x,y)\\ &\iff (x,r')\leq(y,t\rightarrow s),\end{align*}   completing the proof.
 \end{proof}
By the ordinal sum decomposition theorem of continuous t-norms \cite{KMP00} and the argument of Proposition \ref{Archi vs reciprocal}, one readily verifies the following conclusion.
\begin{cor}\label{characeterizing R} Let ${\sf Q}=([0,1],\&,1)$, where $\&$ is a continuous t-norm on $[0,1]$. Then, a $\sf Q$-category $X$ has property {\rm (R)} if and only if it satisfies the following condition: for all $x,y\in X$, if $X(x,y)\geq p$  for some idempotent element $p>0$, then $X(x,y)=1$. \end{cor}


\begin{lemma}\label{mainlemma}Suppose that ${\sf Q}=({\sf Q},\with,k)$ is a continuous and integral quantale such that $\with$ distributes over non-empty meets. Let $X$ be a ${\sf Q}$-category; let $a$ be an element of $X$    and let $(x_\lambda,r_\lambda)_{\lambda\in D}$ be a directed subset of $\mathbf{B}X $ for which $\bigvee_{\lambda\in D}r_\lambda=1$. Consider the statements: \begin{enumerate}[label=\rm(\arabic*)]\setlength{\itemsep}{0pt} \item $a$ is a Yoneda limit  of $(x_\lambda)$. \item   $(a,1)$ is a join of $(x_\lambda,r_\lambda)_{\lambda\in D}$.\end{enumerate} Then, $(1)$ implies $(2)$. Further, if $X$ has property {\rm(R)} and $\mathbf{B} X$ is directed complete,  $(2)$ also implies $(1)$.
\end{lemma}

\begin{proof}That (1) implies (2) follows from Lemma \ref{directed set is forward Cauchy} and Lemma \ref{join of directed set in BX}. Now,   assume that $X$ has property {\rm(R)}, $\mathbf{B} X$ is directed complete, and that $(a,1)$ is a join of $(x_\lambda,r_\lambda)_{\lambda\in D}$. We show that $a$ is a   Yoneda limit  of $(x_\lambda)$; that means, for all $y\in X$, \[X(a,y)=\bigvee_{\lambda\in D}\bigwedge_{\mu\geq\lambda}X(x_\mu,y).\]

Fix $\lambda\in D$. Since for all $\mu\geq\lambda$, \begin{align*}r_\lambda\with X(a,y)&\leq r_\mu\with X(a,y) 
 \leq X(x_\mu,a)\with X(a,y)
 \leq X(x_\mu,y),\end{align*} it follows that $$r_\lambda\with X(a,y)\leq \bigwedge_{\mu\geq\lambda}X(x_\mu,y),$$  hence  \begin{align*}X(a,y)&=\bigvee_{\lambda\in D}r_\lambda\with X(a,y)  \leq \bigvee_{\lambda\in D}\bigwedge_{\mu\geq\lambda}X(x_\mu,y).\end{align*}

For the converse inequality, let $$t=\bigvee_{\lambda\in D}\bigwedge_{\mu\geq\lambda}X(x_\mu,y).$$ We wish to show that $t\leq X(a,y)$. We may assume that $t>0$.

It is clear that  $(x_\lambda, t\with r_\lambda)_{\lambda\in D}$ is a directed subset of $\mathbf{B} X$, hence has a join, say $(z,s)$. We claim that $(z,s)\cong(a,t)$; that is, $(z,s)\leq (a,t)$ and $(a,t)\leq (z,s)$.
Since $(a,t)$ is an upper bound of the directed set $(x_\lambda, t\with r_\lambda)_{\lambda\in D}$, it follows that $(z,s)\leq (a,t)$; in particular $0<s\leq t$. Since $\bigvee_{\lambda\in D}r_\lambda=1$, we may assume that  all $r_\lambda$ are large enough. Since $(z,s)$ is a join of $(x_\lambda, t\with r_\lambda)_{\lambda\in D}$, then $(x_\lambda,t\with r_\lambda)\leq(z,s)$ for all $\lambda\in D$, then $(x_\lambda, r_\lambda)\leq(z,t\rightarrow s)$ for all $\lambda\in D$ because $X$ has property {\rm(R)}, and then  $(a,1)\leq(z,t\rightarrow s)$ because $(a,1)$ is a join of $(x_\lambda,r_\lambda)_{\lambda\in D}$. Therefore,   $t=s$ and $(a,1)\leq (z,1)$, hence  $(a,t)\leq (z,t)=(z,s).$  This proves that $(z,s)\cong(a,t)$.

For each $\lambda\in D$, since  \begin{align*}r_\lambda\with t&= r_\lambda\with\bigvee_{\gamma\in D}\bigwedge_{\mu\geq\gamma}X(x_\mu,y) \\
&= r_\lambda\with\bigvee_{\gamma\geq\lambda}\bigwedge_{\mu\geq\gamma}X(x_\mu,y)\quad\quad\text{($D$ is directed)} \\
&\leq \bigvee_{\gamma\geq\lambda}\bigwedge_{\mu\geq\gamma}r_\lambda\with X(x_\mu,y) \\
&\leq \bigvee_{\gamma\geq\lambda}\bigwedge_{\mu\geq\gamma}r_\mu\with X(x_\lambda,x_\mu)\with X(x_\mu,y)\\ &\leq X(x_\lambda,y),\end{align*}   it follows that $(x_\lambda, t\with r_\lambda)\leq (y,1)$. Thus, $(y,1)$ is an upper bound of the set $(x_\lambda, t\with r_\lambda)_{\lambda\in D}$, hence $(a,t)\leq(y,1)$, and then $t\leq X(a,y)$. \end{proof}


\begin{proof}[Proof of Theorem \ref{main}] (i) This is   $(1)  \Rightarrow (2)$ in Lemma \ref{mainlemma}.

(ii)  We show that every forward Cauchy weight $\phi$ of $X$ has  a colimit.  By Lemma \ref{characterization of Cauchy ideal}, there is a directed subset $(x_\lambda,r_\lambda)_{\lambda\in D}$   of $\mathbf{B}X $ such that $\bigvee_{\lambda\in D}r_\lambda=1$ and that  $$\phi=\bigvee_{\lambda\in D}\bigwedge_{\mu\geq \lambda} X(-,x_\mu).$$ By assumption, $(x_\lambda,r_\lambda)_{\lambda\in D}$ has a join, say $(a,1)$. By Lemma \ref{mainlemma}, $a$ is a Yoneda limit of the forward Cauchy net $(x_\lambda)_{\lambda\in D}$, hence a colimit of $\phi$ by Proposition \ref{yoneda limit as colimits}. \end{proof}

\begin{cor}\label{maincor} Let $\with$ be a continuous t-norm on $[0,1]$ and let ${\sf Q}=([0,1],\with,1)$. Then the following are equivalent: \begin{enumerate}[label=\rm(\arabic*)]\setlength{\itemsep}{0pt} \item   $\with$ is Archimedean. \item For each ${\sf Q}$-category $X$,   $X$ is Yoneda complete if and only if  $\mathbf{B}X $ is directed complete.
\end{enumerate} \end{cor}

\begin{proof}$(1)\Rightarrow(2)$ If $X$ is Yoneda complete, then   $\mathbf{B}X $ is directed complete by Proposition \ref{Archi vs FC} and Lemma \ref{join of directed set in BX}. Conversely, if $\mathbf{B}X $ is directed complete, then $X$ is Yoneda complete by  Proposition \ref{Archi vs reciprocal} and Theorem \ref{main}.

$(2)\Rightarrow(1)$  Example \ref{Example-G}.\end{proof}

The following example shows that  in Theorem \ref{main}\thinspace(ii),  the requirement that $\mathbf{B}X $ is directed complete cannot be weakened to that every directed subset $(x_\lambda,r_\lambda)_{\lambda\in D}$ of $\mathbf{B}X $ with $\bigvee_{\lambda\in D}r_\lambda=1$ has a join.

\begin{example} \label{second example} Let ${\sf Q}$ be the quantale ${\sf Q}=([0,1],\with_P,1)$, where $\with_P$ is the product t-norm. By Proposition \ref{Archi vs reciprocal} every $\sf Q$-category has property (R). We claim that there is a  $\sf Q$-category $X$ such that every directed subset $(x_\lambda,r_\lambda)_{\lambda\in D}$ of $\mathbf{B}X $ with $\bigvee_{\lambda\in D}r_\lambda=1$ has a join, but $X$ is not Yoneda complete. Since  $([0,1],\with_P,1)$ is isomorphic to Lawvere's quantale $([0,\infty]^{\rm op},+,0)$, it suffices to construct a quasi-metric space $(X,d)$ such that $(X,d)$ is not Yoneda complete, but every directed subset $(x_\lambda,r_\lambda)_{\lambda\in D}$   of $\mathbf{B} X$ with $\inf_{\lambda\in D}r_\lambda=0$ has a join.

Let $X=\{0\}\cup\{1/n\mid n\geq2\}$. Define a quasi-metric $d$ on $X$ by \[d(x,y)=\begin{cases}1/2 &x=0, y\not=0,\\ \max\{0,y-x\}&{\rm otherwise}.\end{cases}\] Then $(X,d)$ satisfies the requirements.

(i) $(X,d)$ is not Yoneda complete, since the sequence $(1/n)_{n\geq2}$ is forward Cauchy but has no Yoneda limit.

(ii) We show that every directed subset $(x_\lambda,r_\lambda)_{\lambda\in D}$   of $\mathbf{B}X$ with $\inf_{\lambda\in D}r_\lambda=0$ has a join. Since $\inf_{\lambda\in D}r_\lambda=0$, it follows from Lemma \ref{directed set is forward Cauchy} that $(x_\lambda)_{\lambda\in D}$ is a forward Cauchy net of $(X,d)$, hence either  $(x_\lambda)_{\lambda\in D}$ is  eventually constant or  $(x_\lambda)_{\lambda\in D}$ converges to $0$ (in the usual sense). 

Case 1. $(x_\lambda)_{\lambda\in D}$ converges to $0$. In this case we show that $(0,0)$ is a join of $(x_\lambda,r_\lambda)_{\lambda\in D}$. First, since $d(z,0)=0$ for all $z\in X$, then $r_\lambda\geq 0+d(x_\lambda,0)$ for all $\lambda\in D$, hence $(0,0)$ is an upper bound of $(x_\lambda,r_\lambda)_{\lambda\in D}$. Next, assume that $(y,s)$ is an upper bound of   $(x_\lambda,r_\lambda)_{\lambda\in D}$. It is clear that $s=0$, so $r_\lambda\geq d(x_\lambda,y)$ for all $\lambda\in D$. Since $r_\lambda$ converges to $0$ and $(x_\lambda)_{\lambda\in D}$ converges to $0$, then $y=0$. This shows that $(0,0)$ is the only upper bound, hence a join,  of $(x_\lambda,r_\lambda)_{\lambda\in D}$.

Case 2. $(x_\lambda)_{\lambda\in D}$ is  eventually constant. By assumption there is some $a\in X$ and some $\lambda\in D$ such that $x_\mu=a$ whenever $\mu\geq\lambda$. Then it is readily verified that $(a,0)$ is a join of $(x_\lambda,r_\lambda)_{\lambda\in D}$.  \end{example}


However, for standard quasi-metric spaces  (see  \cite[Definition 2.1]{Goubault-Ng}), we have the following conclusion.

\begin{cor}Let $X$ be a  standard quasi-metric spaces. Then, $X$ is Yoneda complete if and only if every directed subset $(x_\lambda,r_\lambda)_{\lambda\in D}$   of $\mathbf{B}X$ with $\inf_{\lambda\in D}r_\lambda=0$ has a join. \end{cor}

\begin{proof}Proposition 2.4 in \cite{Goubault-Ng} shows that if $X$ is a  standard quasi-metric space and every directed subset $(x_\lambda,r_\lambda)_{\lambda\in D}$   of $\mathbf{B}X$ with $\inf_{\lambda\in D}r_\lambda=0$ has a join, then $\mathbf{B}X$ is directed complete. Thus, the conclusion follows from Theorem \ref{main}\thinspace(ii) immediately.   \end{proof}

\section{Conclusion}
Following Lawvere \cite{Lawvere1973}, the study of $\sf Q$-categories is  part of a \emph{generalized pure logic} with $\sf Q$ as the set of \emph{truth-values}. The connection between categorical properties of a $\sf Q$-category and order-theoretic properties of its set of formal balls has received much attention both in mathematics and theoretic computer science.  Corollary \ref{maincor} in this paper shows that the structure of the truth-values, i.e., the structure of the quantale $\sf Q$, also interacts with this connection. This kind of interaction deserves further investigation.


\end{document}